\documentclass[reqno,12pt]{amsart}
\usepackage{latexsym,graphicx,epsfig,graphics}

\usepackage{amsmath, amsthm, amssymb}
\usepackage{pgf,tikz}
\usepackage{url}
\usepackage{verbatim}
\usepackage{enumerate, enumitem}
\usepackage{rotating}
\usepackage{array}
\usepackage{longtable}
\usepackage{calc}
\usepackage{multirow}
\usepackage{hhline}
\usepackage{ifthen}
\usepackage{amsthm}
\usepackage{mathrsfs}
\usepackage{amsfonts}
\usepackage{eucal}
\usepackage{setspace}

\usetikzlibrary{arrows,topaths,calc}
\definecolor{sqsqsq}{rgb}{0.13,0.13,0.13}
\theoremstyle{plain}
\newtheorem{theorem}{Theorem}[section]
\newtheorem{proposition}[theorem]{Proposition}
\newtheorem{corollary}[theorem]{Corollary}
\newtheorem{lemma}[theorem]{Lemma}

\theoremstyle{definition}
\newtheorem{definition}[theorem]{Definition}

\newtheorem{example}[theorem]{Example}
\newtheorem{conjecture}[theorem]{Conjecture}

\numberwithin{equation}{section}
\title{Iteration of the mincut graph operator}

\author{C. Kriel$^{\S}$ \and E. Mphako-Banda$^{\S}$ }
\begin{document}

\date{}
\maketitle
 
\centerline{$^\S$ School of Mathematics,}
\centerline{University of the Witwatersrand, PB 3, WITS, 2050,
RSA.}
\centerline{ Email: christo.kriel@wits.ac.za; eunice.mphako-banda@wits.ac.za}

\begin{abstract}
A graph operator is a mapping $\phi$ which maps every graph $G$ from some class of graphs to a new graph $\phi(G)$. In this paper, we introduce and study the properties of the \emph{mincut} operator, specifically the effects of iteration of the operator. We show that the property of being super edge-connected and regular is both necessary and sufficient for a graph to remain fixed under the mincut operator. Furthermore, we show that no graph diverges under iteration of this operator. We conclude by stating further research questions on the mincut operator.
\par
\vspace{.1in}
\noindent Keywords: Connectivity, edge-cut set, mincut, intersection graph, graph operator.\\
\noindent Subject Classification: 05C40, 05C70, 05C76

\end{abstract}
\doublespacing
\section{Introduction}
\label{sec:intro}
Let $\mathcal{F}=\{S_1, S_2 \ldots , S_n\}$ be any family of sets. The \emph{intersection graph} of $\mathcal{F}$, denoted $\Omega (\mathcal{F})$ is the graph having $\mathcal{F}$ as vertex set with $S_i$ adjacent to $S_j$ if $S_i\cap S_j \neq \emptyset$, see \cite{intersectiongr}. In $1945$, Szpilrajn-Marczewski proved that every graph is an intersection graph, that is, for every graph $G$ there is a family of sets $\mathcal{F}$ such that $G$ is the intersection graph of $\mathcal{F}$, see \cite{graphreps}. In \cite{mingraph}, an intersection graph, the \emph{mincut graph} of a graph, with vertex set the minimum edge-cuts of a graph, such that two vertices in the intersection graph are adjacent if their corresponding minimum edge-cut sets have non-empty intersection, was introduced. Furthermore, it was shown that every graph is a mincut graph, that is, every graph lies in the image of the mincut operator. In this paper, we follow the research programme on graph operators, as introduced by Prisner in the 1995 monograph ``Graph Dynamics''. Thus, we show that graphs that are fixed under the mincut operator are super edge-connected and regular. In addition, we show that all graphs eventually converge under iteration of the mincut operator, that is, after a sufficiently large number of iterations no graph increases in size or order.

\begin{definition}
	Let $G$ be a simple connected graph, then an edge-cut of $G$ is a subset $X$ of $E(G)$, such that $G-X$ is disconnected. An edge-cut of minimum cardinality in $G$ is a \emph{minimum edge-cut} and its cardinality is the edge-connectivity of $G$, denoted $\lambda(G)$. We will call such a minimum edge-cut a \emph{mincut} of $G$. We call a mincut that disconnects a single vertex from the graph a \emph{trivial cut}.
\end{definition}

\begin{definition}\cite{mingraph}
	Let $X=\{X_1, \, X_2, \, \ldots \, X_i\}$ be the set of all mincuts of a simple connected graph $G$. Represent each of the $X_i$ with a vertex $x_i$ such that two vertices $x_i$ and $x_j$ are adjacent if $X_i\cap X_j \neq \emptyset$ and call this intersection graph the \emph{mincut graph} of $G$, denoted by $X(G)$.
	\label{def:XG}
\end{definition}

\begin{example}
	In Figure \ref{fig:mincutgraphs}, we give some examples of the mincut graphs of trees, $T_n$, complete graphs, $K_n$, cycles, $C_n$, and wheels $W_n$.
\end{example}

\begin{figure}[!h]
	\begin{center}
		\includegraphics[width=\textwidth]{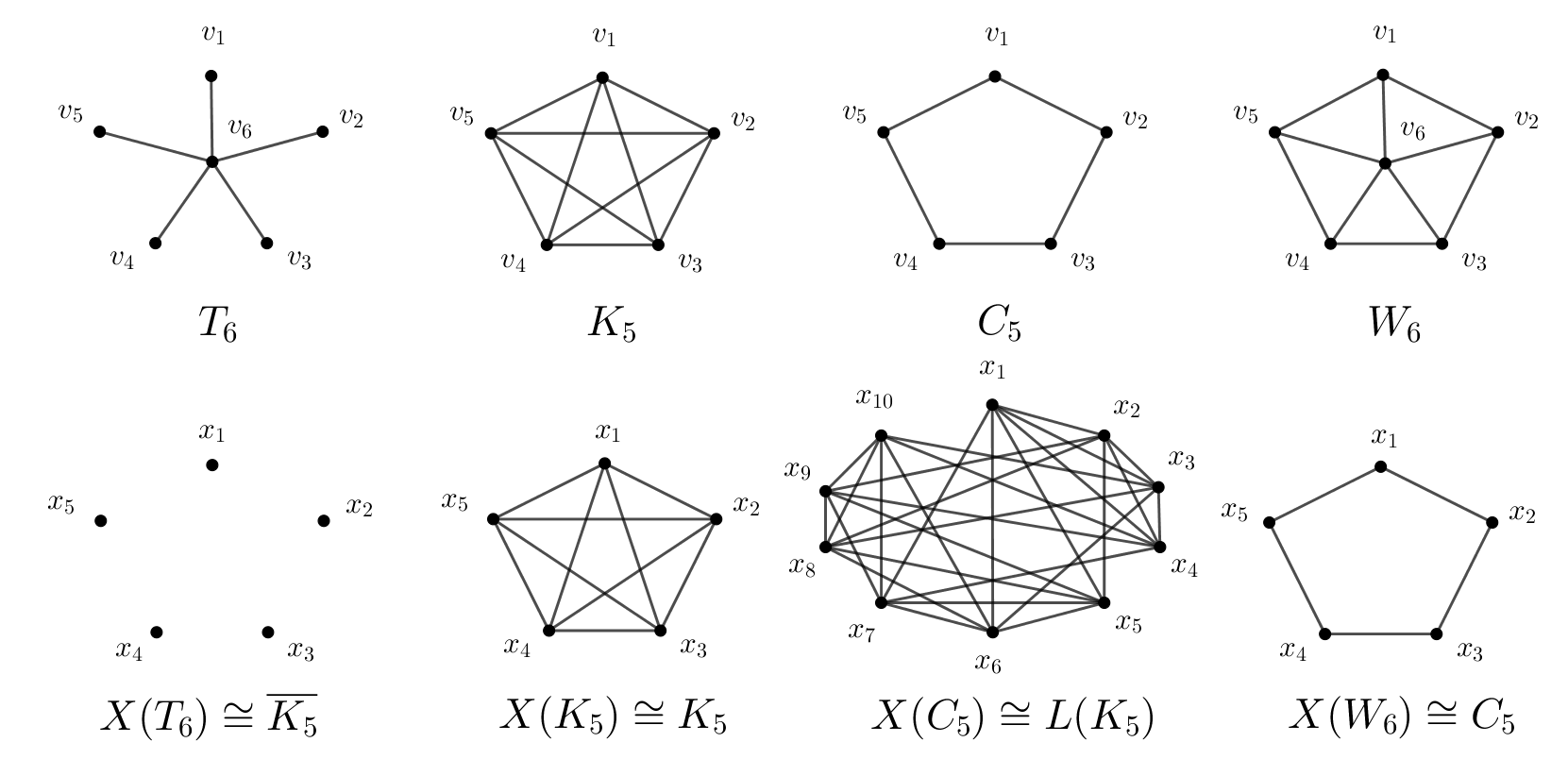}
	\end{center}
	\caption{Mincut graphs of some well-known classes of graphs.}
	\label{fig:mincutgraphs}
\end{figure}

As the line graph of a graph $G$ reflects the mutual positions of the edges, see \cite{prissurv}, so the mincut graph reflects the mutual positions of the minimum edge-cuts.

\begin{definition}\cite{prissurv}
	A graph operator is a mapping $\phi$ which maps every graph $G$ from some class of graphs to a new graph $\phi(G)$.
	\label{def:graphop}
\end{definition}

We define the operator recursively, such that, if $\phi$ is an operator and $k$ a positive integer, then $\phi^1=\phi$ and $\phi^k(G)=\phi(\phi^{k-1}(G))$, for $k\geq 2$, see \cite{grdynamics,prissurv}. In \cite{grdynamics}, Prisner unifies the study of graph operators under such questions as  `Which graphs appear as images of graphs?'; `Which graphs are fixed under the operator?'; `What happens if the operator is iterated?'.

In \cite{mingraph}, the question as to which graphs appear as images of graphs was answered by showing that every graph is a mincut graph of an infinite family of non-isomorphic graphs, that is, every graph has an infinite number of $X$-roots, where a graph $H$ is an $X$-root of $G$ if $X(H)\cong G$. In addition, it was also shown that every graph has infinite depth, that is, given a graph $G_0$, there is some graph $G_{(-1)}$ such that $X(G_{(-1)})\cong G_0$ and some graph $G_{(-2)}$ such that $X(G_{(-2)})\cong G_{(-1)}$. Thus, we can construct a set $\{\ldots,\, G_{(-3)}, \, G_{(-2)}, \, G_{(-1)}, \, G_0\}$ where every $X(G_{(-n)})\cong G_{(-n+1)}$ for all $n\geq 1$ and $X^n(G_{(-n)})\cong G_0$.

In this paper, we investigate Prisner's remaining questions as follows:

\begin{enumerate}
	\item When is a graph isomorphic to its mincut graph?
	\item What happens when the mincut operator is iterated?
\end{enumerate}

\section{Super edge-connected graphs}
\label{sec:superedge}
In this section, we define super-edge connected, or super-$\lambda$, graphs and give a sufficient condition for a mincut graph of a graph $G$ to be isomorphic to $G$.

\begin{definition}\cite{handbookgraphth}
	A graph $G$ is \emph{maximally edge connected} when $\lambda=\delta$, where $\lambda$ is the cardinality of a minimum edge-cut and $\delta$ is the minimum vertex degree of $G$.
	\label{def:maxedge}
\end{definition}

\begin{definition}\cite{handbookgraphth}
	A graph $G$ is super edge-connected, \emph{super-$\lambda$}, if every minimum edge-cut set is \emph{trivial}; that is, consists of the edges incident on a vertex of minimum degree.
	\label{def:superlambda}
\end{definition}

The following proposition states five sufficient conditions for a graph to be super-$\lambda$, see \cite{handbookgraphth}. The first two conditions were proved by Lesniak in 1974, see \cite{lessuperl}.

\begin{proposition}
	\label{prop:suffsuperl}
	Let $G$ be a graph of order $n$, minimum degree $\delta$ and maximum degree $\Delta$. Then $G$ is super-$\lambda$ if any of the following conditions are satisfied:
	\begin{enumerate}
		\item Let $G\neq K_{n/2} \Box K_2$, the cartesian product of $K_{n/2}$ and $K_2$. If for any non-adjacent $u,v\in V(G)$, $\textrm{deg}(u)+\textrm{deg}(v)\geq n,$	then $G$ is super-$\lambda$.
		
		\item If for any non-adjacent $u,v\in V(G)$, $\textrm{deg}(u)+\textrm{deg}(v)\geq n+1,$	then $G$ is super-$\lambda$.
		
		\item If $\delta \geq \lfloor \frac{n}{2} \rfloor +1$, then $G$ is super-$\lambda$.
		
		\item If $G$ has diameter $2$ and contains no complete subgraph $H$ on $\delta$ vertices with $\textrm{deg}_G(v)=\delta$ for all $v\in V(H)$, then $G$ is super-$\lambda$.
		
		\item If $G$ has diameter $2$ and $n>2\delta+\Delta-1,$ then $G$ is super-$\lambda$.
	\end{enumerate}
	
\end{proposition}

\begin{proposition}
	If $G$ is $r$-regular and super-$\lambda$, $G\ncong K_2$, then $X(G)\cong G$.
	\label{prop:X(G)=G}
\end{proposition}

\begin{proof}
	Since $G$ is super-$\lambda$, the mincuts of $G$ are exactly the edges incident on the vertices of minimum degree. Since $G$ is regular, this is true for every vertex of $G$. Hence, there is a one-to-one correspondence between the vertices of $X(G)$ and the vertices of $G$. If two vertices are adjacent in $G$, their corresponding trivial mincuts have non-empty intersection, and hence the adjacencies in $G$ are preserved in $X(G)$. 
\end{proof}

\begin{corollary} Let $K_n$ be the complete graph on $n$ vertices, $K_{n,n}$ the complete bipartite graph with equal vertex partitions and $L(K_n)$ the line graph of the complete graph. If $n>2$, then
	\begin{enumerate}[label=(\roman*)]
		\item $X(K_n)\cong K_n$
		\item $X(K_{n,n})\cong K_{n,n}$
		\item $X(L(K_n))\cong L(K_n).$
	\end{enumerate}
	\label{corol:Xsuperl}
\end{corollary}
\begin{proof}
	If $n=2$, $(i)$ is excluded by Proposition \ref{prop:X(G)=G} and $(iii)$ is excluded since $L(K_2)\cong K_1$. We have $K_{2,2}\cong C_4$ and $X(C_4)\cong L(K_4)$. If $n>2$ all three families of graphs are clearly regular and super-$\lambda$.
\end{proof}

\section{When is a graph isomorphic to its mincut graph?}
\label{sec:isoX}

In this section, we address the question of which graphs are fixed under the mincut operator, that is, remain unchanged under the operator. We show that the property of a graph $G$ being super-$\lambda$ and $r$-regular is not only sufficient but also necessary to have $X(G)\cong G$. We start by examining the structure of mincuts in more detail.

We know that if $X$ is a mincut of a graph $G$, then $G\backslash X$, that is $G$ with all edges belonging to $X$ deleted from $E(G)$, is a disconnected graph with exactly two components, $A$ and $\bar{A}$. We will also refer to the vertex sets of the two components as $A$ and $\bar{A}$, respectively, and write $X=\langle A,\bar{A} \rangle$ to make clear that $X$ is the set of edges connecting the two components. For the following definition see \cite{chandran,lehel}.

\begin{definition}
	Let $X=\langle A,\bar{A} \rangle$ and $Y=\langle B,\bar{B} \rangle$ be two mincuts of a graph $G$. If the vertex sets of their corresponding components have non-empty intersection, that is $A\cap B\neq \emptyset$, then they are either \emph{nested}, i.e. $A\subset B$ or $B\subset A$, or they \emph{overlap} (also called \emph{crossing mincuts}), i.e. $A\cap B$, $\bar{A}\cap B$ and $A\cap \bar{B}$ are \emph{non-empty} and $A\not\subset B$ or $B\not\subset A$.
	\label{def:nestedandcross}
\end{definition}

We also note from \cite{lehel} that two mincuts can overlap only if $\lambda$, the cardinality of the minimum edge cut, is even.

\begin{proposition}[Lehel et al, \cite{lehel}]
	Let $X=\langle A,\bar{A} \rangle$ and $Y=\langle B,\bar{B} \rangle$ be two overlapping mincuts of a connected graph $G$. Then $|\langle \bar{A}\cap \bar{B}, A\cap \bar{B} \rangle |=|\langle \bar{A}\cap \bar{B}, \bar{A}\cap B \rangle |=|\langle A\cap B, A\cap \bar{B} \rangle |=|\langle A\cap B, \bar{A}\cap B \rangle |=\dfrac{\lambda}{2}$, where $\lambda$ is the minimum edge-cut number of $G$. Consequently $A\cup B, \, A\cap B, \, A\cap \bar{B}, \, \bar{A}\cap B$ are mincuts. Moreover, $|\langle \bar{A}\cap \bar{B}, A\cap B \rangle |=|\langle \bar{A}\cap B, A\cap \bar{B} \rangle |=0$.
	\label{prop:crossing}
\end{proposition}

\begin{lemma}[Chandran \& Ram, \cite{chandran}]
	If $X=\langle A,\bar{A} \rangle$ and $Y=\langle B,\bar{B} \rangle$ are a pair of crossing mincuts, then $X\cap Y = \emptyset$.
	\label{lem:crossempty}
\end{lemma}

We illustrate the concepts of nested and crossing mincuts in the following two examples.

\begin{example}
	The graphs in Figure \ref{fig:egnested} are four copies of a simple connected graph $G$ with $\lambda=3$. We note that $G$ has four non-trivial nested mincuts. The diagrams in the figure highlight the four non-trivial mincuts $Z=\langle A,\bar{A} \rangle=\{e_3,\, e_6,\, e_7\},\, V=\langle B,\bar{B} \rangle=\{e_1,\, e_2,\, e_3\}, \, Y=\langle C,\bar{C} \rangle=\{e_1,\, e_4,\, e_5\}$ and $W=\langle D,\bar{D} \rangle=\{e_5,\, e_9,\, e_{10}\}$. The intersections $Z\cap V,\, V\cap Y$ and $Y\cap W$ are all non-empty.
	
	We note $A=\{b,\,c,\,d,\,e\}$, $B=\{a,\,b,\,c,\,d,\,e\}$, $C=\{a,\,b,\,c,\,d,\,e,\,g\}$ and $D=\{a,\,b,\,c,\,d,\,e,\,f,\,g\}$ and that the intersection of the vertex sets of the corresponding components of $Z,\, V,\, Y \textrm{ and } W$ are non-empty. We also have $\lambda$ odd which means that the vertex sets are nested, i.e. $A\subset B\subset C\subset D$.
	
	Equivalently, we consider the complementary vertex sets in each partition and we see that their intersections are non-empty and $\bar{D} \subset \bar{C} \subset \bar{B} \subset \bar{A}$.
	
	The subgraph induced by the four vertices corresponding to these non-trivial mincuts in $X(G)$ is a path $Z\rightarrow V\rightarrow Y\rightarrow W$.
	\label{eg:nested}
\end{example}

\begin{figure}[!h]
	\begin{center}
		\includegraphics{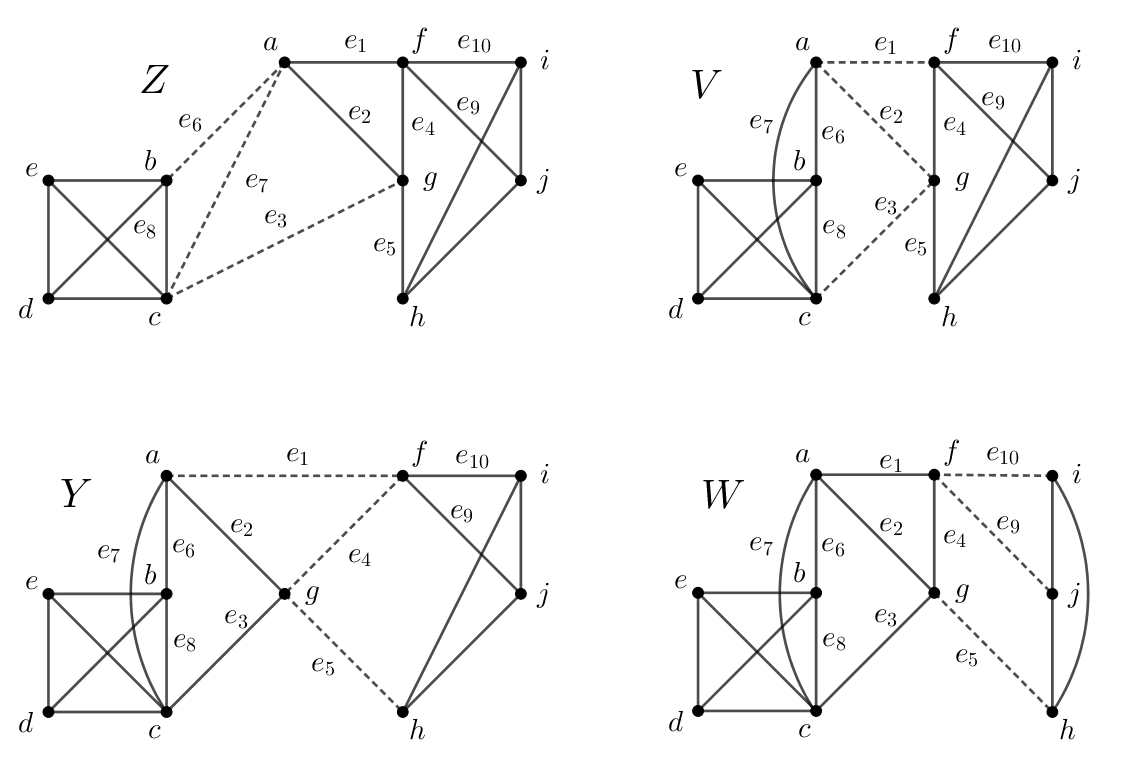}
	\end{center}
	\caption{Nested mincuts $Z,\, V,\, Y \textrm{ and } W$ of a graph $G$, where $\lambda$ is odd.}
	\label{fig:egnested}
\end{figure}

\begin{example}
	The graphs in Figure \ref{fig:egcrossing} are five copies of a simple connected graph $G$ with $\lambda=4$. We note that $G$ has five non-trivial mincuts. The diagrams in the figure highlight the five non-trivial mincuts $Z=\langle A,\bar{A} \rangle,\, U=\langle B,\bar{B} \rangle, \, Y=\langle C,\bar{C} \rangle$, $V=\langle B\cap C, \bar{B}\cup \bar{C} \rangle$ and $W=\langle B\cup C,\bar{B}\cap \bar{C} \rangle$. Consider $U$ and $Y$ and note that $B=\{a,\,b,\,c,\,d,\,k,\,l,\,m\}$, $\bar{B}=\{e,\,f,\,g,\,h,\,i,\,j\}$, $C=\{a,\,b,\,c,\,e,\,k,\,l,\,m\}$ and $\bar{C}=\{d,\,f,\,g,\,h,\,i,\,j\}$. Thus we have $B\cap C$, $\bar{B}\cap C$, $B\cap \bar{C}$ and $\bar{B}\cap \bar{C}$ all non-empty and by Definition \ref{def:nestedandcross}, $U$ and $Y$ are overlapping or crossing mincuts. Since $U$ and $Y$ overlap, Proposition \ref{prop:crossing} indicates that we should also have $\langle B\cap C,\bar{B}\cup \bar{C} \rangle$ and $\langle B\cup C,\bar{B}\cap \bar{C} \rangle$ as mincuts and these are, respectively, $V$ and $W$. Furthermore, $B\cap \bar{C}=\{d\}$, and $\bar{B}\cap C=\{e\}$ are both vertices of degree four and hence their incident edge sets are mincuts.
	
	By Lemma \ref{lem:crossempty}, if $U$ and $Y$ are crossing mincuts then we should have $U\cap Y=\emptyset$ and indeed we see that $U=\{e_3,\,e_4,\,e_5,\,e_6\}$ and $Y=\{e_1,\,e_2,\,e_7,\,e_8\}$. Mincut $Z$ is a nested mincut since $A\subset (B\cap C)$. We also have $A\subset B$ and $A\subset C$. Similarly $V$ is nested since $(B\cap C)\subset B$ and $(B\cap C)\subset C$. We also have $B\subset (B\cup C)$ and $C\subset (B\cup C)$ so $U$ and $Y$ are nested with respect to $W$, but $B\not\subset C$ and $C\not\subset B$ since $U$ and $Y$ overlap. Considering the edge sets we have $Z=\{e_1, \,e_4,\,e_9,\,e_{10}\}, \, V=\{e_1, \,e_2,\,e_3,\,e_4\}$ and $W=\{e_5, \,e_6,\,e_7,\,e_8\}$. Hence, the subgraph induced by the vertices in $X(G)$ corresponding to $V,\,U,\,Y,\,W$ is a cycle and the vertex corresponding to $Z$ is adjacent to the vertices $V,\,U,\,Y$.
	\label{eg:crossing}
\end{example}

\begin{figure}[!h]
	\begin{center}
		\includegraphics{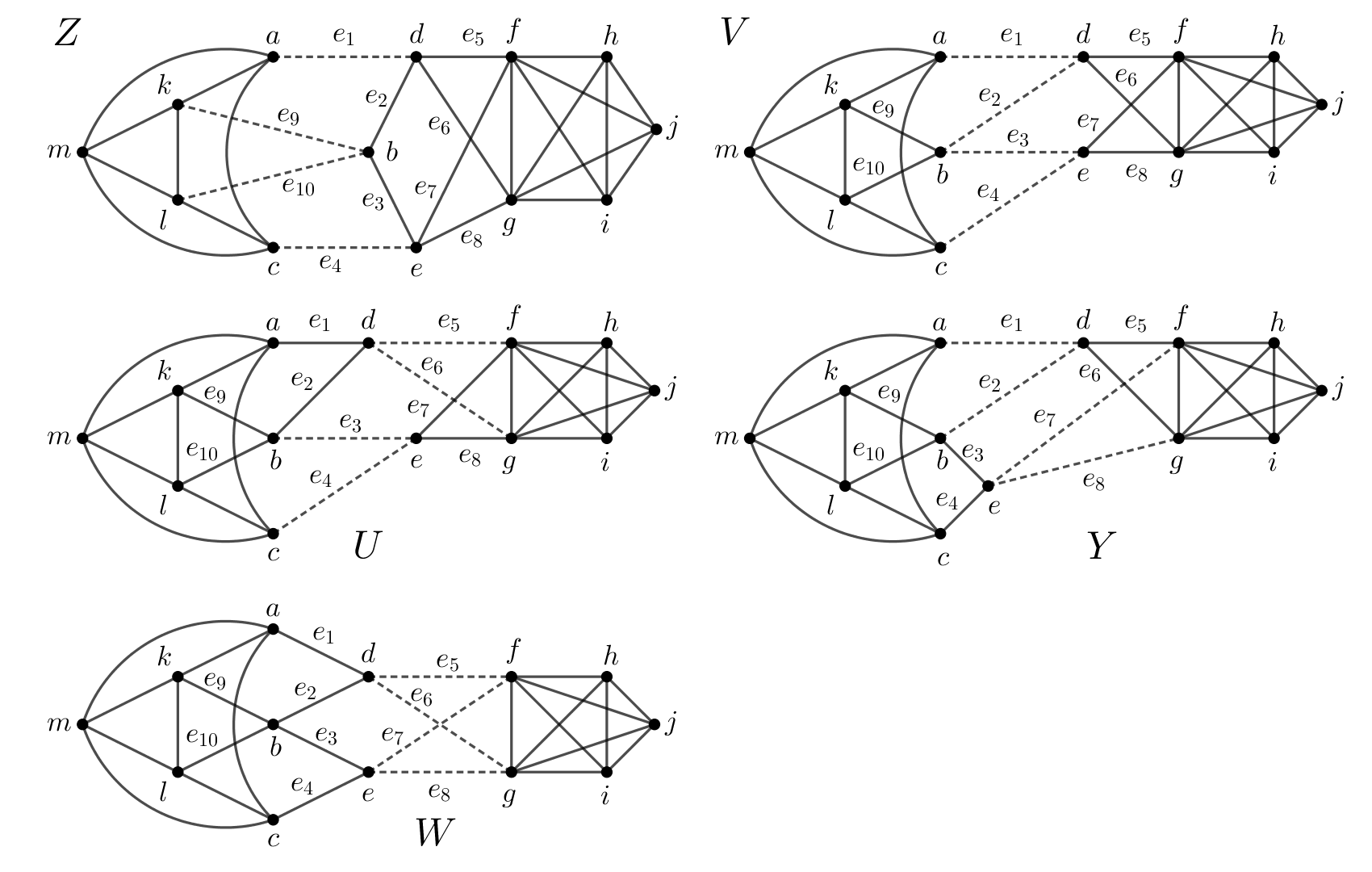}
	\end{center}
	\caption{Nested and overlapping mincuts $Z,\, V,\, U, \, Y \textrm{ and } W$ of a graph $G$, where $\lambda$ is even.}
	\label{fig:egcrossing}
\end{figure}

When we consider a mincut $X$ in terms of the vertex sets of the two components of $G\backslash X$, then there is an important subset of vertices of each of the components, namely the vertices incident on the edges of $X$. We define vertices of attachment as follows, specifically with regard to mincuts, although the definition can be generalised for any subgraph of $G$, as originally defined by Tutte, see \cite{tutconnectivity,tutte1984graph}.

\begin{definition}
	Let $X$ be a mincut of $G$ and $A$ and $\bar{A}$ the two components of $G\backslash X$. A \emph{vertex of attachment} of $A$ (resp. $\bar{A}$) in $G$ is a vertex of $A$ (resp. $\bar{A}$) that is incident with some edge of $G$ that is not an edge of $A$ (resp. $\bar{A}$).
	\label{def:vertatt}
\end{definition}

We will denote the vertices of attachment of the two components $A$ and $\bar{A}$ of $G\backslash X$ as $V_A$ and $V_{\bar{A}}$ respectively. Also, we will denote the neighbourhood of a vertex $v$ restricted to the neighbours in one of the two components as $N_A(v)$ or $N_{\bar{A}}(v)$ respectively.

\begin{example}
	Consider $Z=\langle A,\bar{A} \rangle=\{e_1,\,e_4,\,e_9,\,e_{10}\}$ in Figure \ref{fig:egcrossing}, then $V_A=\{a,\,c,\,k,\,l\}$ and $V_{\bar{A}}=\{b,\,d,\,e\}$.
	\label{eg:vertatt}
\end{example}

\begin{definition}
	Let $v\in V(G)$ be a vertex of a graph $G$ such that $deg(v)=\lambda$, that is, the set of edges incident on $v$ is a (trivial) mincut. We define $v$ to be a \emph{trivial vertex}.
	\label{def:trivvert}
\end{definition}

\begin{lemma}
	Let $X=\langle A, \,\bar{A}\rangle$ and $Y=\langle B,\, \bar{B}\rangle$ be two non-trivial mincuts of a graph $G$ with $A,\, \bar{A}$ and $B, \, \bar{B}$ the vertex sets of the components of $G\backslash X$ and $G\backslash Y$ respectively. If $X\cap Y \neq \emptyset$ then either $A\cap B\neq \emptyset$ or $A\cap \bar {B}\neq \emptyset$.
	\label{lem:XYcomp}
\end{lemma}

\begin{proof}
	Let $e=uv\in X\cap Y$ be an edge belonging to both $X$ and $Y$, then $u$ and $v$ are vertices of attachment, one of $A$ and the other of $\bar{A}$, since $e\in X$. Now $e\in Y$ and hence $u$ and $v$ are also vertices of attachment of $B$ and $\bar{B}$. Therefore we must have $A\cap B \neq \emptyset$ or $A\cap \bar{B} \neq \emptyset$.
\end{proof}

Note that the converse to Lemma \ref{lem:XYcomp} is not necessarily true.

Recall that two graphs, $G$ and $X(G)$, are isomorphic if there is a bijective function $\psi : V(G) \to V(X(G))$ such that two vertices $v_i$ and $v_j$ are adjacent in $G$ if and only if $\psi(v_i)=x_i$ and $\psi(v_j)=x_j$ are adjacent in $X(G)$.

The $X$-operator maps a mincut $X_i\subseteq E(G)$ to a vertex $x_i\in V(X(G))$. We will say that $x_i$ corresponds to $X_i$. But if $G\cong X(G)$ then there is a vertex $v_i$ mapped to $x_i$ and thus we can say that this $v_i$ corresponds to $X_i$ and to $x_i$.

In the rest of this section we will use $v$ for vertices in $G$, $X$ for mincuts in $G$ and $x$ for vertices in $X(G)$ and use the same subscript if two objects correspond. For example, we will write $v_i\sim x_i$ or $v_i \sim X_i$. 

The key behind the isomorphism mapping $G\to X(G)$ is that if $v_j\sim x_j$ and $v_j$ is adjacent to $v_i$ in $G$, then $x_j\in N_{X(G)}(x_i)$ and $X_j\cap X_i\neq \emptyset$.

\begin{lemma}
	Let $v_i,\, v_j\in V(G)$ be trivial vertices with incident edge sets $X_i$ and $X_j$. If $v_iv_j \in E(G)$, then $x_i\sim X_i$ and $x_j\sim X_j$ are vertices of $X(G)$ such that $x_ix_j\in E(X(G))$. If $v_iv_j\notin E(G)$ then $x_ix_j\notin E(X(G))$.
	\label{lem:trivadjtriv}
\end{lemma}
\begin{proof}
	$X_i$ and $X_j$ are the edge sets incident on $v_i$ and $v_j$ respectively and are mincuts since the two vertices are both trivial. Let $e$ be an edge incident on both $v_i$ and $v_j$ then $v_iv_j\in E(G)$ and $e$ is an edge belonging to both $X_i$ and $X_j$, so $X_i\cap X_j\neq \emptyset$ and $x_ix_j\in E(X(G))$.
	Suppose $v_iv_j\notin E(G)$ and $e$ is any edge incident on $v_i$, then $e$ is not incident  on $v_j$ and hence $e\notin X_j$. Therefore, $X_i\cap X_j=\emptyset$ and $x_ix_j\notin E(X(G))$.
\end{proof}

\begin{lemma}
	Let $X_k$ be a non-trivial mincut with $A$ and $\bar{A}$ the components of $G\backslash X_k$ and $x_k\sim X_k$ where $x_k\in V(X(G))$. If $v_i$ is a trivial vertex and $v_i\in V_A$ (or $V_{\bar{A}}$), then $X_i\cap X_k \neq \emptyset$ and $x_ix_k\in E(X(G))$. If $v_i \notin V_A$ or $v_i \notin V_{\bar{A}}$, then $X_i\cap X_k=\emptyset$ and $x_ix_k\notin E(X(G))$.
	\label{lem:trivintnontriv}
\end{lemma}
\begin{proof}
	If $v_i\in V_A$ then $v_iv_j\in E(G)$ for some $v_j\in V_{\bar{A}}$ which means there is an edge $e\in X_k$ incident on $v_i$ and $v_j$. But every edge incident on $v_i$ is in $X_i$ and $X_i$ is a mincut. So $X_i\cap X_k\neq \emptyset$ and $x_ix_k\in E(X(G))$.
	Let $v_i\in A$ and $v_i\notin V_A$. Then $v_i$ is adjacent to no $v_j\in V_{\bar{A}}$ and hence there is no $e\in X_i$ such that $e\in X_k$. Hence $X_i\cap X_k=\emptyset$ and $x_ix_k\notin E(X(G))$.
\end{proof}

\begin{example}
	The diagram in Figure \ref{fig:egcrossing} shows the non-trivial mincut $Z=\langle A,\bar{A}\rangle=\{e_1, \,e_4,\,e_9,\,e_{10}\}$. Vertex $a\in V_A$ is trivial since $\lambda(G)=deg(a)=4$ and we see that $X_a$ shares edge $e_1$ with $Z$. However, the trivial vertex $m\notin V_A$ and we note that $X_m$ shares no edge with $Z$.
	\label{eg:trivintnontriv}
\end{example}

\begin{lemma}
	Let $G\cong X(G)$ and $X_k$ be a non-trivial mincut. Then there is $v_k\in V(G)$ such that $deg(v_k)>\lambda$.
	\label{lem:nontrivexists}
\end{lemma}
\begin{proof}
	Suppose there is no such $v_k$, then all vertices of $G$ have degree $\lambda$ and their incident edge-sets are mincuts. Since $X_k$ is a non-trivial mincut this gives $|V(G)|+1$ mincuts in $G$ and hence $|V(X(G))|>|V(G)|$ so $G\ncong X(G)$.
\end{proof}

	
	

\begin{lemma}
	Let $X_k=\langle A,\bar{A} \rangle$ be a non-trivial mincut of a connected graph $G$ with edge connectivity $\lambda$. If either $|V_A|\textrm{ or }|V_{\bar{A}}|=1$ then $G\ncong X(G)$.
	\label{lem:cutvertex}
\end{lemma}
\begin{proof}
	Let $v_k\in A$ be the only vertex of attachment of $A$. Since $X_k$ is a mincut we must have $|N_{\bar{A}}v_k|=\lambda$. If $|N_A(v_k)|=\emptyset$ then $X_k$ is a trivial cut. Therefore, we must have some $v_l\in N_A(v_k)$ and $v_k$ is non-trivial. But $|N_A(v_k)|\geq \lambda$, otherwise we have an edge set with cardinality less than $\lambda$ that will disconnect $v_k$ from $A$, contradicting the minimality of $\lambda$.
	
	If $|N_A(v_k)|=\lambda$, then there exists a mincut, $X_l=\langle B,\bar{B} \rangle$ say, such that $v_k\in \bar{B}$ and $|V_{\bar{B}}|=1$. Hence, $X_l\cap X_k=\emptyset$ and $X(G)$ will be disconnected. In fact, no $X_i\in E(G[A])$, the edge set of the induced subgraph on $A$, can intersect $X_k$ and $X(G)$ will be disconnected.
	
	Suppose there is no mincut, trivial or otherwise, of $G$ in the subgraph induced by $A$, $G[A]$. Then $X_k$ is mapped to $x_k\in V(X(G))$ but none of the other vertices in $N_A(v_k)$ are mapped to a corresponding component in $X(G)$ and $X(G)\ncong G$.
\end{proof}

\begin{theorem}
	$G\cong X(G)$ if and only if $G$ is super-$\lambda$ and $r$-regular.
	\label{th:Xfixed}
\end{theorem}

\begin{proof}
	By Proposition \ref{prop:X(G)=G}, if $G$ is super-$\lambda$ and $r$-regular then $G\cong X(G)$.
	
	To prove necessity we need to show that if $G\cong X(G)$ then $G$ is super-$\lambda$ and $r$-regular. We will prove the contrapositive, that is, if $G$ does not have the property of being \emph{both} super-$\lambda$ and $r$-regular then $G\ncong X(G)$. Hence, we need to prove that if $G$ is not super-$\lambda$ or not $r$-regular then $G\ncong X(G)$.
	
	Suppose $G$ is super-$\lambda$ but not $r$-regular.
	
	If $G$ is super-$\lambda$ then all mincuts are trivial, that is only edge sets incident on trivial vertices are mincuts. But $G$ is not $r$-regular and we conclude that there must be vertices $v\in V(G)$ such that $deg(v)>\lambda$. Hence, we have $|X|<|V(G)|$, the number of mincuts is less than the number of vertices and consequently $|V(X(G))|<|V(G)|$, so $G\ncong X(G)$.
	
	It remains to show that if $G$ is not super-$\lambda$ but $r$-regular or $G$ is not super-$\lambda$ and not $r$-regular, then $G\ncong X(G)$. We show that $G$ not being super-$\lambda$ is sufficient for $G\ncong X(G)$.
	
	Suppose $G\cong X(G)$ and $G$ is not super-$\lambda$.
	
	If $G$ is not super-$\lambda$ then not all mincuts are trivial. Let $X_k$ be a non-trivial mincut with $A$ and $\bar{A}$ the two components of $G\backslash X_k$ and $V_A$ and $V_{\bar{A}}$ their respective sets of vertices of attachment. By Lemma \ref{lem:nontrivexists} there exists a non-trivial vertex $v_k\in V(G)$ such that $v_k\sim X_k$. We have a number of cases of how the position of some vertex $v$ and $N(v_k)$ relate to the sets of vertices of attachment corresponding to $X_k$. We start with the supposition that $\lambda=\delta$ and \emph{all} the neighbours of $v_k$ are trivial.
	
	\begin{enumerate}[label=Case \Roman*:, wide, labelwidth=!, labelindent=0pt]
		\item \noindent Suppose $v_k\in V_A$ and we have $v_i\in N_A(v_k)$, but $v_i \notin V_A$. Then, by Lemma \ref{lem:trivintnontriv} $X_i\cap X_k= \emptyset$ and $x_i \notin N(x_k)$ in $X(G)$. The same holds if $v_k\notin V_A$ and $v_i\notin V_A$. We conclude that any $v_i\in N_A(v_k)$ must be in $V_A$.
		\item Suppose $v_k\notin V_A$ is adjacent to $v_i\in V_A$ and $v_iv_j\in E(G)$ where $v_j\in V_{\bar{A}}$ is trivial. Then $v_kv_j\notin E(G)$, but $X_k\cap X_j\neq \emptyset$ and hence $x_kx_j\in E(X(G))$.
	\end{enumerate}
	
	We conclude that we must have $v_k\in V_A$ and all trivial vertices $v\in N_A(v_k)$ must be in $V_A$. Furthermore, this implies that all other neighbours of $v_k$ are in $V_{\bar{A}}$. 
	
	\begin{enumerate}[label=Case \Roman*:, wide, labelwidth=!, labelindent=0pt, resume]
		\item Suppose $v_k\in V_A$, and all $v\in N(v_k)$ are trivial and are either in $V_A$ or $V_{\bar{A}}$. All $v_i\in V_A,\,V_{\bar{A}}$ are necessarily neighbours of $v_k$ since their mincuts intersect $X_k$ and hence their corresponding vertices $x_i$ are neighbours of $x_k$.
		
		$X_k$ is a mincut and hence $|X_k|=\lambda$. Consider the neighbours of $v_k$. Each vertex $v\in N_{\bar{A}}(v_k)$ contributes at least one edge to $X_k$ (edges incident on $v_k$) and each vertex $v\in N_A(v_k)$ contributes at least one edge to $X_k$ since they are in $V_A$. Therefore
		\[\lambda=|X_k|\geq |N_{\bar{A}}(v_k)|+|N_A(v_k)|.
		\]
		But $deg(v_k)=|N_{\bar{A}}(v_k)|+|N_A(v_k)|$ contradicting the non-triviality of $v_k$.
	\end{enumerate}
	
	From cases I to III, if all $v\in N(v_k)$ are trivial, then $G\ncong X(G)$. So, suppose we have some $v_l\in N(v_k)$ such that $deg(v_l)>\lambda$.
	
	\begin{enumerate}[label=Case \Roman*:, wide, labelwidth=!, labelindent=0pt, resume]
		\item Suppose $v_l\in V_A$. The neighbours of $v_k$ contribute at least one edge to $X_k$ regardless of whether they are trivial or not as long as they are in $V_A$ or $V_{\bar{A}}$ and since $deg(v_k)>\lambda$ we have $|X_k|>\lambda$ and then $X_k$ cannot be a mincut.
	\end{enumerate}
	
	From cases I to IV we can now conclude the following:
	
	Firstly, if $v_i$ and $v_j$ are trivial and $v_k$ is non-trivial with $v_k\in A$, $v_i\in N_A(v_k)$ and $v_j\in V_{\bar{A}}$, then $v_i,\, v_k \in V_A$. Furthermore, if all $v\in N(v_k)$ are trivial then $X_k$ is not a mincut.
	
	Secondly, if all $v\in N(v_k)$ are in $V_A$ or $V_{\bar{A}}$, then $X_k$ is not a mincut. Importantly, this also applies if $V_A=A$. So there must be some $v_l\in N(v_k)$ such that $v_l\in A$ but $v_l \notin V_A$. If $v_l$ is trivial then $v_l\in V_A$, so $v_l$ must be non-trivial.
	
	\begin{enumerate}[label=Case \Roman*:, wide, labelwidth=!, labelindent=0pt, resume]	
		\item Let $v_l\in N_A(v_k)$ but $v_l\notin V_A$ and $deg(v_l)>\lambda$. This is different to the first case since here $v_l\in N_A(v_k)$, but $v_l$ is not trivial. By Lemmas \ref{lem:nontrivexists} and \ref{lem:trivcortriv} there exists a non-trivial mincut $X_l=\langle B,\bar{B} \rangle$ such that $X_l\sim v_l$. If $G\cong X(G)$ then  $X_l\cap X_k\neq \emptyset$. Let $V_B$ and $V_{\bar{B}}$ be the respective sets of vertices of attachment of $B$ and $\bar{B}$. Assume that there are some $v\in N_B(v_l)$ and $v\in V_{\bar{B}}$ that are trivial, so $v_l\in V_B$. 
		
		From Lemma \ref{lem:crossempty} we know that if $X_k$ and $X_l$ are crossing mincuts, then $X_k\cap X_l=\emptyset$, so $X_l$ is nested with respect to $X_k$ and $B\subset A$, so $|B|<|A|$. If $A\subset B$ we simply interchange $v_k$ and $v_l$ and the reasoning is the same.
		
		Applying the same reasoning following cases $I$ to $IV$ to $v_l$, we must have some non-trivial $v_m\in N_B(v_l)$ such that $v_m\notin V_B$ and by Lemma \ref{lem:nontrivexists} we have a non-trivial mincut $v_m\sim X_m=\langle C, \bar{C}\rangle$. Also $X_m\cap X_l\neq \emptyset$ which implies that $C\subset B$ with $|C|<|B|$. Again, assume that there  are some $v\in N_C(v_k)$ and $v\in V_{\bar{C}}$ that are trivial, so $v_m\in V_C$. 
		
		We continue the process of identifying some non-trivial, adjacent $v\notin V_{(\cdot)}$, which implies the existence of a nested non-trivial $X_{(\cdot)}$ and each successive vertex set is smaller than the previous one.
		
		Eventually we must get to a non-trivial $v_z\sim X_z=\langle Z,\bar{Z}\rangle$, such that $v_z\in N_Y(v_y)$, for some $v_y\sim X_y=\langle Y, \bar{Y}\rangle$, with $v_z\notin V_Y$ and $Z=V_Z$. If $v_z$ is trivial then it must be in $V_Y$. Recall from Lemma \ref{lem:cutvertex} that if $X_z$ is non-trivial then $V_Z>1$. Since $Z=V_z$, all neighbours of $v_z$ are in $V_Z$ or $V_{\bar{Z}}$ and hence $X_z$ cannot be a mincut.
		
		Consequently, we cannot have a non-trivial mincut for every non-trivial vertex and we conclude that $|V(G)|>|V(X(G))|$ so $G\ncong X(G)$.
		
		\item Finally, it remains to show that if $G$ is not super-$\lambda$ and $G$ has no trivial mincuts, then $G\ncong X(G)$. Let $X_k=\langle A, \bar{A} \rangle$ be a nontrivial mincut of $G$ where $\lambda(G)<\delta$ so $G$ has no trivial mincuts. Assume $X(G)\cong G$, then we know from Lemma \ref{lem:cutvertex} that $|V_A|>1$. Let $v_k$ be one of the vertices in $V_A$ and $v_l\in N_A(v_k)$ such that $v_l\notin V_A$. There must be such a $v_l$ otherwise all neighbours of $v_k$ are in $V_A$ or $V_{\bar{A}}$ which would mean that $X_k$ is not a mincut. Hence we also have $v_l\sim X_l=\langle B, \bar{B} \rangle$. If $G\cong X(G)$ we must have $X_l\cap X_k\neq \emptyset$ so $X_l$ is nested and $B\subset A$. As before, continuing the process of identifying adjacent vertices and corresponding non-trivial cuts leads to successively smaller vertex sets and eventually we will have a situation where $X_z=\langle Z,\bar{Z} \rangle$ is not a mincut, leading to the conclusion that $|X|<|V(G)|$ and hence $G\ncong X(G)$.
		
	\end{enumerate}
	
	We conclude that if $G$ is not super-$\lambda$ we cannot have $G\cong X(G)$.
	
\end{proof}

\section{Convergence of all graphs under the $X$-operator}
\label{sec:Xconv}

In \cite{mingraph}, it was shown that every graph is a mincut graph and that for every graph $G$, $X(G)$ is unique. But this implies that $G$ has infinitely many $X$-roots and is of infinite $X$-depth. In this section, we clarify what we mean by convergence or divergence under an operator. We show that no graph diverges under $X(\cdot)$ and postulate that in most cases iteration of the operator leads to the null graph, $K_0$, with no vertices or edges. For definitions on operator concepts, we refer the reader to \cite{grdynamics,prissurv}.

\begin{example}
	\label{eg:Xiter}
In Figure \ref{fig:Xiter}, iteration of the first graph leads to $K_0$. We note that if $G$ is disconnected then $\lambda=0$, we have no mincuts and hence $X(G)\cong K_0$. Since the mincut graph of a graph is unique, each of the infinite roots of the first graph in sequence I will lead to $K_0$, each of the infinite roots of the first graph  in sequence II will lead to $L(K_4)$. But $L(K_4)$ is $6$-regular and super-$\lambda$ and hence $X(L(K_4))\cong L(K_4)$. The Cartesian product $K_3\Box K_2$ in sequence III is maximally connected, that is $\lambda(K_3\Box K_2)=\delta(K_3\Box K_2)$ and is $3$-regular. However, it is not super-$\lambda$ since there is one non-trivial mincut, the three edges connecting the two copies of $K_3$ in the product, and it intersects all the trivial mincuts. Hence $X(K_3\Box K_2)\cong K_1\vee K_3\Box K_2$, the join of $K_1$ and $K_3\Box K_2$. But $K_1\vee K_3\Box K_2$ is super-$\lambda$ with $K_3\Box K_2$ the induced subgraph on vertices of minimum degree and hence $X(K_1\vee K_3\Box K_2)\cong K_3\Box K_2$. Each of the infinite roots of $K_3\Box K_2$ in sequence III will lead to the circuit $K_3\Box K_2\leftrightarrow  K_1\vee K_3\Box K_2$. We say that these graphs \emph{converge} to the respective graphs and/or circuits.

\begin{figure}[!h]
	\begin{center}
		\includegraphics[width=\textwidth]{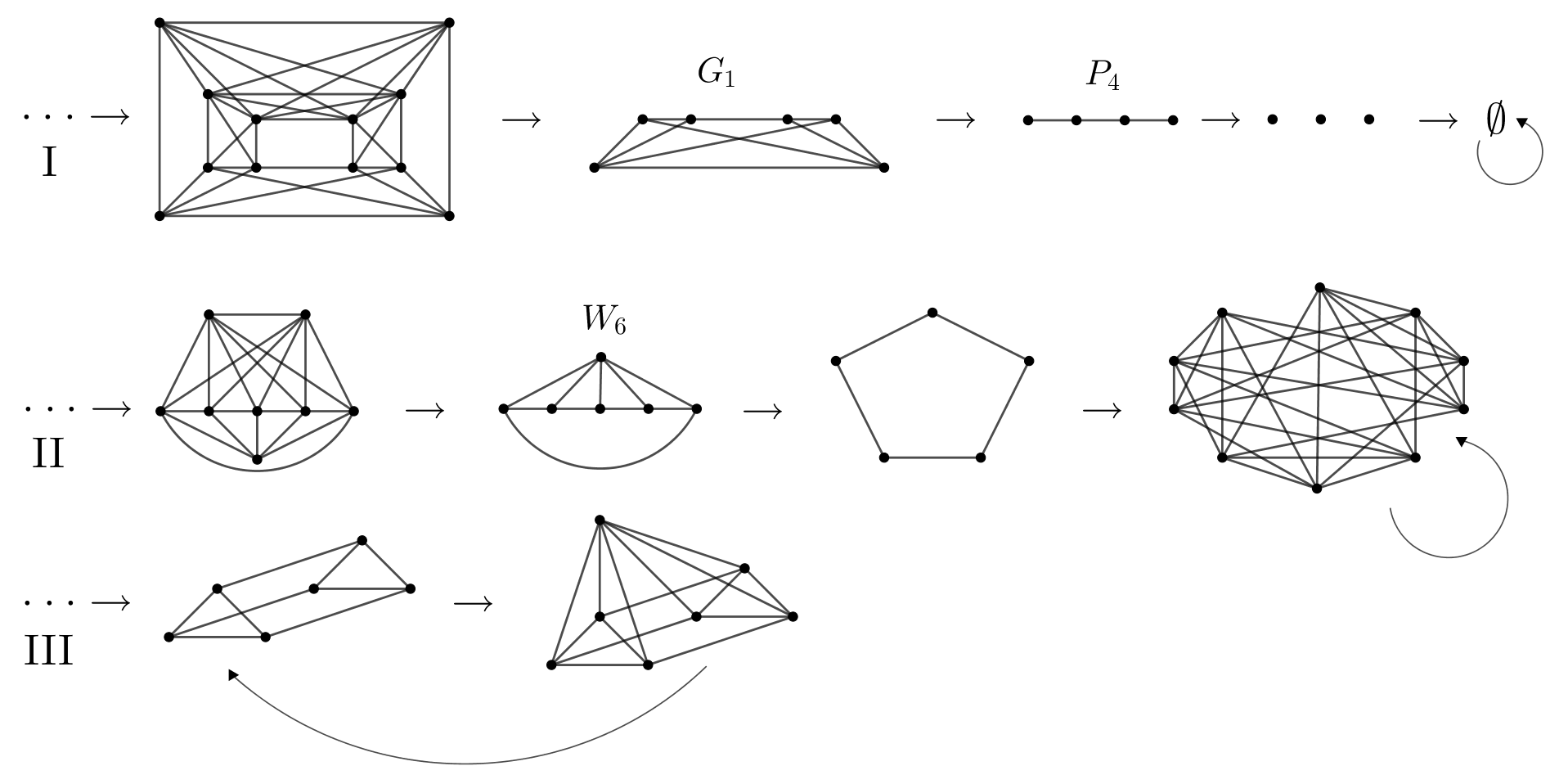}
	\end{center}
	\caption{Effect of the iteration of the mincut operator on some graphs.}
	\label{fig:Xiter}
\end{figure}

\end{example} 

\begin{definition}
	A graph $G$ is $\phi$-\emph{periodic} under an operator $\phi$ if $G\cong \phi^p(G)$ for some integer $p\geq 1$. The smallest such integer $p$ is called the $\phi$-\emph{period} of $G$.
	\label{def:period}
\end{definition}

\begin{definition}
	Let $G$ be a graph and $\phi$ an operator on $G$. $G$ is \emph{convergent} if $\{\phi^n(G)|n\in \mathbb{N} \}$ is finite. That is, $G$ is $\phi$-convergent if and only if some iterated $\phi$-graph is $\phi$-periodic. If $G$ is not convergent then it is \emph{divergent}.
	\label{def:conv}
\end{definition}

\begin{definition}
	A \emph{circuit} is any set of the form $\{G,\,\phi(G),\,\phi^2(G),\, \ldots,\, \phi^{p-1}(G),\, \phi^p(G)\cong G\}$.
	\label{def:circuit}
\end{definition}

Let $G$ be some convergent graph, then there is a unique circuit $M$ such that $\phi^n(G)\in M$ for sufficiently large $n$, see \cite{grdynamics}.

Do any graphs diverge under $X(\cdot)$? We usually distinguish $\phi$-convergent from $\phi$-divergent graphs by means of some parameter that increases under $\phi$, see \cite{prissurv}. With respect to the $X$-operator we consider the order, $|V|$, and size, $|E|$, of successive graphs under iteration of $X(\cdot)$ and the implications of increasing these numbers with respect to the number of mincuts of a graph, $|X|$, and the number of non-empty intersections, $|\cap X_i|$, of edge sets belonging to $X$.

\begin{definition}
	The \emph{edge density}, D, of a graph $G$ on $n$ vertices is the ratio of edges to the total number of edges possible, that is
	\[
	D=\frac{|E|}{\binom{n}{2}}=\frac{2|E|}{n(n-1)}.
	\]
\end{definition}

We recall from the definition of the mincut graph that the number of mincuts $|X|$ in $G$ is the number of vertices in $X(G)$ and the number of intersections of the mincuts in $G$ is the number of edges in $X(G)$. Now, the maximum number of edges in a simple connected graph is $\binom{n}{2}$, where every choice of two vertices are adjacent, and when this happens the graph is complete, which is fixed under the $X$-operator, since $X(K_n)\cong K_n$, and the graph converges. Furthermore, we know that the maximum number of mincuts in a graph of order $n$ is $\binom{n}{2}$ and this bound is tight for simple graphs when $G\cong C_n$, see \cite{dinic,lehel}, and $X(C_n)\cong L(K_n)$ which is fixed under $X(\cdot)$, since $X(L(K_n)\cong L(K_n)$. We also have other upper bounds on the number of mincuts of a graph depending on whether $\lambda$ is even or odd, see \cite{lehel}.

\begin{equation*}
	|X|\leq\begin{cases}
		\frac{2n^2}{(\lambda+1)^2}+\frac{(\lambda-1)n}{\lambda+1}, & \text{if $\lambda\geq 4$ and $\lambda$ is an even integer},\\
		(1+\frac{4}{\lambda+5})n, & \text{if $\lambda>5$ and $\lambda$ is an odd integer}.
	\end{cases}
\end{equation*}

Considering these bounds we should expect that for large $\lambda$, $|X|$ approaches $n$ and we have an upper limit on the number of vertices of $X(G)$ and that $X(G)$ must at some point either become fixed, periodic or disconnected. We note with reference to this last case that if at any stage of the iteration we have a mincut (subset of mincuts) that does not intersect with any other mincut (shares no element with any other subset of mincuts) then $X(G)$ is disconnected and converges to the null graph $K_0$.

The Cartesian product $K_n\Box K_2$ is maximally connected, that is $\lambda(K_n\Box K_2)=\delta(K_n\Box K_2)=n$ and is $n$-regular. However, it is not super-$\lambda$ since there is one non-trivial mincut, the $n$ edges connecting the two copies of $K_n$ in the product, and it intersects all the trivial mincuts. By the same reasoning used in Example \ref{eg:Xiter} for $K_3\Box K_2$ we note that $K_n\Box K_2$ is $2$-periodic. 
 
 We recall from Proposition \ref{prop:suffsuperl} that if $deg(u)+deg(v)\geq n$ and $G\ncong K_{n/2}\Box K_2$ then $G$ is super-$\lambda$. In the light of convergence we will use this sufficient condition since we know that $K_{n/2}\Box K_2$ is $2$-periodic, and hence converges by Definitions \ref{def:conv} and \ref{def:circuit}.

\begin{theorem}
	Let $G$ be a simple graph of order $n$ and size $m$ with minimum degree $\delta$ and edge-connectivity $\lambda$, then $G$ converges under iteration of the $X$-operator, that is $X^i(G)$ converges for sufficiently large $i$.
	\label{th:Xconv}
\end{theorem}

\begin{proof}
	We proceed by showing that no graph diverges under iteration of $X(\cdot)$. In order for $G$ to diverge we must have an increase in $n$, $m$ or both.
	
	Suppose $n$ increases but $m$ does not. If we get to a point where $n=m$ we have a cycle and the next iteration is $L(K_n)$, fixed under the operator. If $m=n-1$ we have a tree and the next iteration is disconnected leading to the null graph. If $m<n-1$ then $G$ is disconnected and $X(G)\cong K_0$. In other words, if the graph becomes sufficiently \emph{sparse} then $X^i(G)$ converges to the null graph.
	
	Suppose $m$ increases but $n$ does not. We get to $m=\binom{n}{2}$, the maximum number of edges possible. But then $G$ is complete and fixed under the operator. In other words, if the graph becomes sufficiently \emph{dense} $X^i(G)$ becomes fixed.
	
	Hence, we need both $n$ and $m$ to increase under iteration of the operator in order for the graph to diverge. In order for $n$ to increase we need $|X|$ to increase, and in order for $m$ to increase we need the number of intersecting mincuts $|\cap X_i|$ to increase as we iterate the operator. In short, for each $G_i$ in the iteration sequence we need $|X|\geq n$ and $|\cap X_i|\geq m$, where one of the inequalities has to be strict, but for this to happen we need $\lambda$ to increase. If $\lambda$ does not increase the size of each $X_i$ stays the same (or decreases) and eventually $n$ increases but $m$ does not. If $\lambda$ increases, $\delta$ increases since $\kappa \leq \lambda \leq \delta$, as originally proved by Whitney. Now if $\delta\geq \frac{n}{2}$, or equivalently, $deg(u)+deg(v)\geq n$ for any $u,v\in V(G)$, then $G$ is super-$\lambda$ or $G\cong K_{n/2}\Box K_2$, by Proposition \ref{prop:suffsuperl}.
	
	If $G\cong K_{n/2}\Box K_2$, then it is not diffcult to show that $G$ is $2$-periodic and hence converges.
	
	If $G$ is super-$\lambda$ and $r$-regular, then $G$ is fixed, by Theorem \ref{th:Xfixed}.
	
	If $G$ is super-$\lambda$ but not $r$-regular, then $|X|<|V(G)$, see the relevant part of the proof to Theorem \ref{th:Xfixed}. That is $|V(X(G))|<|V(G)|$, so $n$ decreases and the graph cannot diverge.
\end{proof}

\section{Further Discussion}

In this section, we conclude by introducing and exploring further topics and questions raised by the properties and characteristics of the mincut graph.

\subsection{Convergence to the null graph}

If $X(G)$ is disconnected, $X(X(G))$ is the null graph. That is, if there is any of the $X_i$, or subset of intersecting $X_i$, that does not intersect with any other of the $X_i$, or subset of intersecting $X_i$, then the operator will yield a disconnected graph and, hence, iteration leads to the null graph. In general, it would therefore seem that, in most cases, repeated application of $X$ should at some stage yield a disconnected graph $G$ with $\lambda(G)=0$ and hence $X(G)$ is the null graph with no vertices and no edges. 

\begin{conjecture}[Convergence to null graph]
	Let $G$ be a simple connected graph and $X(\cdot)$ the mincut operator. Then $X^k(G)\rightarrow K_0$ except under a finite number of conditions.\\
\end{conjecture}

\subsection{Connected mincut graphs}

A question that naturally arises and that is closely connected to the previous question, is under what conditions will $X(G)$ be connected. That is, what (if any) are the properties of $G$ that will ensure that every $X_i$, or subset of intersecting $X_i$, intersects with some other $X_i$, or intersecting subset of $X_i$. Furthermore, does the edge-connectivity of $X(G)$ tell us anything about the connectivity of $G$?

What leads to $G$ not converging to the null graph? If adjacent vertices are trivial then their mincuts intersect. If a trivial vertex is an element of the vertices of attachment of a non-trivial cut then their intersections are nonempty. It would therefore seem that a high proportion of trivial vertices to the total number of vertices guarantees a connected mincut graph. See for example the cycle: large cycles have sparse edge density but their mincuts have high edge density.

\subsection{Periodicity}

We know that there are graphs that $1$-periodic (fixed) under $X(\cdot)$ and we characterised such graphs as super-$\lambda$ and $r$-regular in this paper. We also know that there are graphs of period $2$. For example, the Cartesian product $K_n\Box K_2$ is such a class of graphs as was shown in the discussion on convergence. But are there other graphs that are $2$-periodic? Are there graphs that are have a higher periodicity than $2$? What would be the sufficient properties for such graphs to be $k$-periodic, for $k\geq 2$?

\bibliographystyle{abbrv}
\bibliography{cutset}
		
\end{document}